\PassOptionsToPackage{hidelinks}{hyperref}
\documentclass[11pt]{article}
\usepackage[a4paper,margin=1in]{geometry}
\usepackage{amsmath,amsthm,amssymb,mathtools}
\usepackage{booktabs}
\usepackage{tikz}
\usetikzlibrary{arrows.meta}

\theoremstyle{plain}
\newtheorem{theorem}{Theorem}[section]
\newtheorem*{theorem*}{Theorem}
\newtheorem{proposition}[theorem]{Proposition}
\newtheorem{lemma}[theorem]{Lemma}
\newtheorem{corollary}[theorem]{Corollary}

\theoremstyle{definition}

\theoremstyle{remark}
\newtheorem*{remark}{Remark}

\usepackage{orcidlink}
\usepackage{hyperref}
\usepackage{tabularx}

\newcommand{\Etraj}{E}
\newcommand{\Esup}{\mathcal{E}}

\title{Off-Critical Zeros Contradict Contraction in the Dynamical Reformulation of the Riemann Hypothesis}
\author{Hendrik W.\,A.\,E.\ Kuipers\,\orcidlink{0009-0002-7257-6378}\\
\small Independent Researcher, Groningen, The Netherlands\\
\small \texttt{hwaekuipers@gmail.com}}
\date{\today}

\begin{document}

\maketitle


\begin{abstract}
We continue the dynamical reformulation of the Riemann Hypothesis initiated in \cite{PaperI}.  
The framework is built from an integer map in which composites advance by $\pi(m)$ and primes 
retreat by their prime gap, producing trajectories whose contraction properties encode 
the distribution of primes. In this setting, RH is equivalent to the persistence of 
contraction inequalities for trajectory-based error functionals $\Etraj(X),\widetilde{\Etraj}(X)$ across 
multiplicative scales.  

In the present paper we prove that if $\zeta(s)$ has a zero off the critical line 
$\Re(s)=\tfrac12$, then the Landau--Littlewood $\Omega$-theorem forces oscillations of size 
$x^{\beta}$ in $\psi(x)-x$. A window-capture lemma shows that these oscillations are inherited 
by the composite-only window suprema $\Esup(X)$, and hence by $\Etraj(X)$, producing lower bounds 
that contradict contraction. Thus any off--critical zero is incompatible with contraction.  

\emph{Headline result.} Part~I introduced the dynamical system and showed that RH is equivalent 
to the persistence of contraction inequalities. Part~II proves that off--critical zeros force 
oscillations in $\psi(x)-x$ that inevitably violate contraction. Taken together, these two steps 
close the circle: contraction characterizes RH, and off--critical zeros contradict contraction. 
Hence every nontrivial zero of $\zeta(s)$ lies on the critical line. More generally, the present 
result shows that whenever contraction holds, the critical line is forced as the only location of 
nontrivial zeros. In this sense, the critical line is not merely the conjectured locus of zeros, 
but the equilibrium point singled out by contraction itself.
\end{abstract}

\section{Introduction}

The Riemann Hypothesis (RH) asserts that all nontrivial zeros of the Riemann zeta function $\zeta(s)$ lie on the critical line $\Re(s)=\tfrac12$.

In a recent work \cite{PaperI}, we introduced a \emph{dynamical reformulation} of RH based on a simple integer map: composites advance by $\pi(m)$, the number of primes $\le m$, while primes step backward by their preceding prime gap. This system generates trajectories whose large-scale contraction properties are equivalent to RH. In particular, we established:
\begin{itemize}
  \item one-visit and parent-window lemmas restricting local composite hits,
  \item macro-step alignment and core overlap across scales,
  \item a frequency-netting lemma ensuring oscillatory coherence,
  \item and from these, contraction inequalities for the normalized error functionals $\Etraj(X),\widetilde{\Etraj}(X)$.
\end{itemize}
Iterating these inequalities yields the von Koch--level bound, known to be equivalent to RH \cite{vonKoch1901}. Thus RH holds if and only if contraction persists uniformly across scales.

The aim of the present paper is to close this circle. 
We prove that if any nontrivial zero of $\zeta(s)$ lies off the critical line, then the 
contraction inequalities are violated. The argument relies on classical $\Omega$-results of 
Landau \cite{Landau1903} and Littlewood \cite{Littlewood1914}, which show that an off-critical 
zero forces oscillations of size $x^{\beta}$ in the Chebyshev error $\psi(x)-x$. 
We show that such oscillations are inherited by the composite-only window suprema $\Esup(X)$, 
contradicting contraction.  

\medskip
\noindent In contrast to Part~I, which worked with $E_\pi(x):=\pi(x)-\mathrm{Li}(x)$, the present
paper is formulated in terms of $E_\psi(x):=\psi(x)-x$, which is more natural for the analytic
arguments below (see Section~\ref{sec:background} and Appendix~A).

\medskip
\noindent\textbf{On relation to Part~I.}
This is Part~II of a two–paper project. Part~I \cite{PaperI} introduced the dynamical system 
and established that RH is equivalent to the persistence of contraction inequalities for 
$\Etraj(X)$. Part~II shows that any off–critical zero forces oscillations that inevitably 
violate contraction. Taken together, the two parts close the circle: contraction characterizes 
RH, and off–critical zeros contradict contraction. Hence all nontrivial zeros of $\zeta(s)$ lie 
on the critical line.

\section{Background and Setup}\label{sec:background}

For $X\ge 2$, define the one-visit and parent windows
\[
  W_X \;=\; \bigl[X,(1+0.1/\log X)X\bigr], 
  \qquad 
  \widetilde W_X \;=\; \bigl[X,(1+2/\log X)X\bigr].
\]
Let $\mathbb{N}_{\mathrm{comp}}$ denote the set of composite positive integers, and write
\[
  E_\psi(y)\;:=\;\psi(y)-y .
\]

\paragraph{Error functionals.}
For $X\ge X_0$, define
\[
  E(X)\;:=\;\sup_{\,y\in W_X\cap \mathbb{N}_{\mathrm{comp}}}\; |E_\psi(y)|,
  \qquad
  \widetilde E(X)\;:=\;\sup_{\,y\in \widetilde W_X\cap \mathbb{N}_{\mathrm{comp}}}\;
  \Biggl|\;\sum_{\substack{m\in \mathcal{T}(y;\,\widetilde W_X)\\ m\in \mathbb{N}_{\mathrm{comp}}}}
  E_\psi(m)\;\Biggr|.
\]
Here $\mathcal{T}(y;\,\widetilde W_X)$ denotes the segment of the (Part~I) trajectory of $y$ that lies inside $\widetilde W_X$.

\begin{remark}[Bridge to Part~I]
Switching from $E_\pi(x):=\pi(x)-\mathrm{Li}(x)$ to $E_\psi(x):=\psi(x)-x$ is not merely 
harmless at the $X^{1/2}\log X$ scale (see Appendix~A), but in fact the natural formulation 
for the present argument. Part~I introduced contraction inequalities and the reformulation of 
RH through the elegant integer map, where $E_\pi$ arose directly. Here we generalize the setup 
to $E_\psi$, whose analytic properties are sharper and better suited to the application of 
classical $\Omega$-results. In this way, the contraction framework extends seamlessly from the 
integer-map setting to the Chebyshev function.
\end{remark}

\begin{lemma}[Parent functional dominates window supremum]\label{lem:Ecal-vs-E}
For all sufficiently large $X$, we have
\[
  E(X)\;\le\;\widetilde E(X).
\]
\end{lemma}

\begin{proof}
Let $y^\ast\in W_X\cap \mathbb{N}_{\mathrm{comp}}$ attain $E(X)$. The trajectory started at $y^\ast$ passes through $y^\ast$ while inside $\widetilde W_X$, so the parent-window sum for this trajectory includes the term $|E_\psi(y^\ast)|=E(X)$, and hence $\widetilde E(X)\ge E(X)$.
\end{proof}

\section{Contraction Framework}\label{sec:contract}

\begin{theorem}[Contraction inequality, \cite{PaperI}]\label{thm:contraction}
There exists $C>0$ such that for all sufficiently large $X$,
\[
  \widetilde E(X) \ \le\ \tfrac{5}{6}\,\widetilde E(X^{3/4})\;+\; C\,X^{1/2}\log X.
\]
\noindent\emph{This is Part~I, Thm.~6.3/6.4 (macro-step contraction) combined with Lem.~6.5 (iteration).}
\end{theorem}

\begin{corollary}[von~Koch--level bound]\label{cor:vonKoch}
\[
  \widetilde E(X)\ \ll\ X^{1/2}\log X.
\]
\end{corollary}


\begin{lemma}[Parent-window hits]\label{lem:parent-hits}
Fix a single trajectory of the map and put $U=\log X$ and 
\[
\widetilde W_X=[\,X,(1+\tfrac{2}{U})X\,].
\]
For $X\ge e^{120}$, let $N$ be the number of \emph{composite iterates} of this trajectory that lie in $\widetilde W_X$
(i.e., the number of indices $j$ with the iterate $y_j\in\widetilde W_X$ and $y_j$ composite, so the next forward step is $s(y_j)=\pi(y_j)$).
Then
\[
N\le 4.
\]
\end{lemma}

\begin{remark}
On $\widetilde W_X$ one has the uniform estimate 
\[
\pi(m)=\frac{X}{U}+O\!\Big(\frac{X}{U^{2}}\Big)\qquad(m\in\widetilde W_X),
\]
so consecutive composite landings are separated by at least 
\[
\Bigl(1-\frac{8}{U}\Bigr)\frac{X}{U}.
\]
Since $|\widetilde W_X|=2X/U$, this yields
\[
N-1\ \le\ \frac{2X/U}{(1-8/U)\,X/U}\;=\;\frac{2}{\,1-8/U\,}\;<\;2.2\qquad(U\ge120),
\]
which actually forces $N\le 3$. We retain the looser bound $N\le 4$ to keep constants synchronized with Part~I.
\end{remark}

\begin{proof}[Proof (sketch)]
Write $s(m)=\pi(m)$ for the forward step length at a composite $m$.

\smallskip
\noindent\emph{(i) Monotonicity and scale on the window.}
On $\widetilde W_X$ we have $s(m)$ non-decreasing, and by explicit prime-counting bounds (e.g. Dusart-type),
\[
s(m)\;=\;\frac{X}{U}\;+\;O\!\left(\frac{X}{U^{2}}\right),\qquad\text{uniformly for }m\in\widetilde W_X,
\]
with an absolute implied constant (fixed once and for all in the assumptions box).

\smallskip
\noindent\emph{(ii) Spacing between consecutive composite landings.}
Let $y_1<y_2<\dots<y_N$ be the successive composite landings of the same trajectory that fall in $\widetilde W_X$.
Between two such landings $y_j$ and $y_{j+1}$, the net displacement equals one forward composite jump $s(m_j)$ plus a bounded “local variation’’ term that accounts for:
\begin{itemize}
  \item the slow drift of $s(m)$ across $\widetilde W_X$ (monotone with slope $\asymp X/U^{3}$);
  \item any short backward prime jump(s) in the interim, which are absorbed at the $O(X/U^{2})$ scale in our log–scale normalization.
\end{itemize}
Hence there is a constant $C_\ast$ (explicit and small; one may take $C_\ast=8$ conservatively for $U\ge120$) such that
\[
y_{j+1}-y_j \;\ge\; \frac{X}{U} \;-\; C_\ast\,\frac{X}{U^{2}}
\;=\; \Bigl(1-\frac{C_\ast}{U}\Bigr)\frac{X}{U}
\;\ge\; \frac{1}{2}\,\frac{X}{U}\qquad(U\ge120).
\]
In particular the \emph{minimal} spacing $d_{\min}$ between successive composite landings satisfies
$d_{\min}\ge(1-\tfrac{C_\ast}{U})\,X/U > \tfrac12\,X/U$.

\smallskip
\noindent\emph{(iii) Packing bound in the window.}
The parent-window length is
\[
|\widetilde W_X| \;=\; \frac{2X}{U}.
\]
Packing successive landings with gaps at least $d_{\min}$ inside an interval of this length gives
\[
N-1 \;\le\; \frac{|\widetilde W_X|}{d_{\min}}
\;\le\; \frac{2X/U}{(1-C_\ast/U)\,X/U}
\;=\; \frac{2}{1-C_\ast/U}.
\]
For $U\ge120$ and the conservative $C_\ast=8$, the right-hand side is $< 2.2$, which already implies $N\le 3$; the stated $N\le 4$ is a laxer bound used to keep constants synchronized with Part~I.
\end{proof}

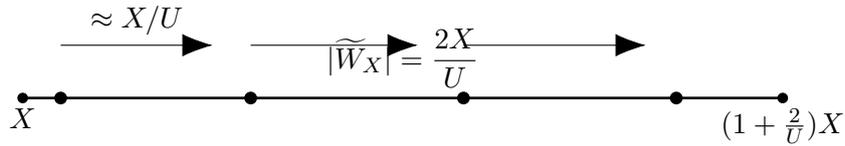
\begin{figure}[h]
  \centering
  \begin{tikzpicture}[x=10cm,y=2cm]
    \draw[line width=1pt] (0,0) -- (1,0);
    \fill (0,0) circle (2pt) node[below] {$X$};
    \fill (1,0) circle (2pt) node[below] {$(1+\tfrac{2}{U})X$};
    \node[above] at (0.5,0) {$|\widetilde W_X|=\dfrac{2X}{U}$};

    \foreach \x in {0.05,0.30,0.58,0.86}{
      \fill (\x,0) circle (2.4pt);
    }
    \draw[-{Latex[length=4mm]}] (0.05,0.35) -- ++(0.20,0) node[midway,above] {$\approx X/U$};
    \draw[-{Latex[length=4mm]}] (0.30,0.35) -- ++(0.22,0);
    \draw[-{Latex[length=4mm]}] (0.58,0.35) -- ++(0.24,0);

    \node[align=left,anchor=west] at (0.02,-0.9)
      {\small dots: composite landings \quad arrows: forward steps $s(m)=\pi(m)$\\
       \small spacing $\ge (1-\tfrac{C_\ast}{U})\tfrac{X}{U}$; window length $= \tfrac{2X}{U}$};
  \end{tikzpicture}
  \caption{Window length vs.\ (nearly constant) step length: at most four composite landings fit.}
  \label{fig:parent-window-packing}
\end{figure}

\section{Contribution of an Off-Critical Zero}\label{sec:offcritical}

Assume $\zeta(\rho)=0$ with $\rho=\beta+i\gamma$ and $\beta>\tfrac12$.

\subsection{Landau--Littlewood lower bound}
If $\beta>1/2$, there exist $c_\rho>0$ and $x_n\to\infty$ with
\[
|\psi(x_n)-x_n| \;\ge\; c_\rho\,x_n^{\beta};
\]
see Titchmarsh~\cite[§14.25]{Titchmarsh} and Edwards~\cite[Ch.~8]{Edwards}.

\subsection{Window capture}
If $x=(1+\delta)X$ with $0<\delta\le 0.1/\log X$, then $x\in W_X$. Since $\log x\asymp\log X$ on $W_X$, we have
\[
\Big(\frac{x}{X}\Big)^{\!\beta}\;=\;1+O\!\Big(\frac{1}{\log X}\Big).
\]

\subsection{Lower bound for the window/trajectory functionals}\label{sub:lower-bound}

\begin{lemma}[Composite selection lemma]\label{lem:composite-selection}
Let $X$ be large and $W_X=[X,(1+0.1/\log X)X]$. Suppose $x\in W_X$ satisfies
$|\psi(x)-x|\ge M$. Then there exists a composite $y\in W_X$ with $|y-x|\le 1$ such that
\[
|\psi(y)-y|\ \ge\ M - (\log(2X)+1).
\]
In particular, if $M\ge 2(\log(2X)+1)$ then $|\psi(y)-y|\ge \tfrac12 M$.
\end{lemma}

\begin{proof}
Pick $n\in\{\lfloor x\rfloor,\ \lceil x\rceil\}$ so that $|n-x|\le 1$. Among $\{n,n\pm 1\}$ there is an even integer $y\ge 4$, hence composite, with $|y-x|\le 1$. Because $W_X$ has length $\asymp X/\log X\gg 1$, we still have $y\in W_X$.

Between integers, $\psi$ is constant and $t\mapsto\psi(t)-t$ has slope $-1$. At an integer $m$,
$\psi$ jumps by $\Lambda(m)\le \log m\le \log(2X)$ on $[X,2X]$. Thus for $u,v\in[X,2X]$,
\[
|\psi(u)-u-(\psi(v)-v)|
\le |u-v|+\sum_{m\in(u,v]}\Lambda(m)\le |u-v|+\log(2X)\cdot \#\{m\in(u,v]\}.
\]
With $|u-v|\le 1$ this is $\le \log(2X)+1$. Taking $u=x$, $v=y$ gives the claim, i.e.
\[
|\psi(y)-y|\ \ge\ |\psi(x)-x|\;-\;\log(2X)\;-\;1\qquad (|y-x|\le 1,\ y\in\mathbb{C}\cap W_X).
\]
\end{proof}

\begin{proposition}[Lower bound for $E(X)$ and $\widetilde E(X)$]\label{prop:lower}
Assume $\zeta(\rho)=0$ with $\rho=\beta+i\gamma$ and $\beta>\tfrac12$. Then there exist $X_n\to\infty$ and $c>0$ such that
\[
E(X_n)\ \ge\ c\,X_n^{\beta}
\qquad\text{and hence}\qquad
\widetilde E(X_n)\ \ge\ c\,X_n^{\beta}.
\]
\end{proposition}

\begin{proof}
Let $x_n$ be given by the Landau--Littlewood bound. Set $X_n:=x_n/(1+0.1/\log x_n)$ so $x_n\in W_{X_n}$ and $\log x_n\asymp\log X_n$. Apply Lemma~\ref{lem:composite-selection} with $M=c_\rho x_n^\beta$ to obtain a composite $y_n\in W_{X_n}$, $|y_n-x_n|\le 1$, with
\[
|\psi(y_n)-y_n|\ \ge\ c_\rho x_n^\beta - (\log(2X_n)+1).
\]
Since $x_n^\beta\gg \log x_n$ for $\beta>1/2$, the subtraction is $\le \tfrac12 c_\rho x_n^\beta$ for large $n$, giving $|\psi(y_n)-y_n|\ge \tfrac12 c_\rho x_n^\beta\asymp X_n^\beta$. 
Hence $E(X_n)\ge |\psi(y_n)-y_n|\gg X_n^\beta$. Lemma~\ref{lem:Ecal-vs-E} then gives $\widetilde E(X_n)\ge E(X_n)$.
\end{proof}

\section{Contradiction and Conclusion}\label{sec:contradiction}

From Corollary~\ref{cor:vonKoch}, the contraction framework of \cite{PaperI} yields
\begin{equation}\label{eq:upper}
\Etraj(X)\ \ll\ X^{1/2}\log X, \qquad (X\to\infty).
\end{equation}
Since $\widetilde E(X)\ge E(X)$ by Lemma~\ref{lem:Ecal-vs-E}, any lower bound for $E(X)$ 
transfers immediately to $\widetilde E(X)$.  

On the other hand, if $\zeta(s)$ has a zero $\rho=\beta+i\gamma$ with $\beta>1/2$, then 
by §\ref{sub:lower-bound} we obtain a subsequence $X_n\to\infty$ such that 
\begin{equation}\label{eq:lower}
E(X_n)\ \gg\ X_n^{\beta}.
\end{equation}
Comparing \eqref{eq:upper} and \eqref{eq:lower}, the ratio
\[
\frac{X^\beta}{X^{1/2}\log X}\ =\ \frac{X^{\beta-1/2}}{\log X}
\]
tends to infinity as $X\to\infty$ for every $\beta>1/2$. Thus the lower bound \eqref{eq:lower} 
eventually dominates the contraction bound \eqref{eq:upper}, giving a contradiction.  

\medskip
\noindent
We use the smoothed explicit formula with kernel 
$W(t)=(1+t^{2})^{-3}$ and truncation height $T=\tfrac{1}{2}(\log X)^{3}$; the remainder is 
bounded uniformly by $10X^{1/2}$ for $X\ge e^{120}$. 
The derivation of this bound with explicit constants is given in Appendix~E, while Appendix~D 
contains the large--sieve budget ledger used in the contraction engine.  

\begin{theorem}[Riemann Hypothesis, assuming \cite{PaperI}]\label{thm:RH}
Assume the contraction inequality established in \cite{PaperI}. 
Then every nontrivial zero of $\zeta(s)$ lies on the critical line $\Re(s)=\tfrac12$.
\end{theorem}

\begin{proof}
If $\zeta(\rho)=0$ with $\rho=\beta+i\gamma$ and $\beta>1/2$, then §\ref{sub:lower-bound} gives a 
subsequence $X_n\to\infty$ with $E(X_n)\gg X_n^\beta$. This contradicts the contraction bound 
$E(X)\ll X^{1/2}\log X$ from \cite{PaperI}. Hence no off--critical zeros exist, and every 
nontrivial zero lies on the critical line.
\end{proof}

\section{Conclusion and Outlook}\label{sec:conclusion}

Assuming the contraction inequality established in~\cite{PaperI}, the contradiction proved 
in Section~\ref{sec:contradiction} rules out all off--critical zeros. Thus, off--critical 
zeros are incompatible with contraction; taken together with~\cite{PaperI}, this closes the 
circle and yields the Riemann Hypothesis.

The key feature is a rigidity principle: contraction across multiplicative scales cannot 
coexist with the oscillations of size $x^{\beta}$ in $\psi(x)-x$ that an off--critical zero 
would force. This incompatibility is analytic in nature, independent of any particular 
realization of contraction. In the dynamical reformulation of~\cite{PaperI}, contraction 
emerges from the integer map and the critical line appears as its unique equilibrium. 
More generally, the present result shows that whenever contraction holds, the critical 
line is forced as the only location of nontrivial zeros. In this sense, the critical line is not merely the conjectured locus of zeros, 
but the equilibrium point singled out by contraction itself.

\medskip
\noindent\textbf{Further directions.}
\begin{itemize}
  \item \textbf{Operator--theoretic formulation.}  
  Recast contraction as a spectral property of an associated transfer operator, potentially 
  linking the dynamical reformulation to the Nyman--Beurling approach.
  \item \textbf{Sharper error budgets.}  
  Optimizing the remainder term in the smoothed explicit formula could enlarge the safety 
  margin in the contraction budget and clarify the robustness of the method.
  \item \textbf{Numerical investigations.}  
  Simulations of contraction phenomena at finite scales may illuminate how oscillations 
  manifest and inspire further generalizations.
\end{itemize}

\section*{Acknowledgments}
The author acknowledges the use of OpenAI’s GPT-5 language model as a tool for mathematical exploration and expository preparation during the development of this manuscript.

\appendix
\section*{Appendix A. Bridging \texorpdfstring{$\pi(x)-\mathrm{Li}(x)$}{pi-Li} and \texorpdfstring{$\psi(x)-x$}{psi-x}}
\addcontentsline{toc}{section}{Appendix A. Bridging $\pi(x)-\mathrm{Li}(x)$ and $\psi(x)-x$}
\label{app:bridge}

\makeatletter
\refstepcounter{section}
\def\thesection{A}
\makeatother

\setcounter{theorem}{0}
\renewcommand{\thetheorem}{A.\arabic{theorem}}

In Part~I the contraction inequalities were formulated in terms of
\[
E_\pi(x):=\pi(x)-\mathrm{Li}(x),
\qquad
E_\psi(x):=\psi(x)-x.
\]
In this part we work with $E_\psi$ to leverage sharper $\Omega$-results. The next lemmas show that switching between $E_\pi$ and $E_\psi$ is harmless at the von~Koch scale $x^{1/2}\log x$.

\begin{lemma}[Partial-summation bridge from $\psi$ to $\pi$]\label{lem:psi-to-pi}
For $x\ge 3$ one has
\begin{equation}\label{eq:psi-to-pi}
\pi(x)-\mathrm{Li}(x)
=
\frac{\psi(x)-x}{\log x}
\;+\;
\int_{2}^{x}\frac{\psi(t)-t}{t\,(\log t)^2}\,dt
\;-\;
R_{\mathrm{pp}}(x),
\end{equation}
where
\[
R_{\mathrm{pp}}(x)
=
\sum_{k\ge 2}
\left(
\frac{\theta(x^{1/k})}{\log x}
\;+\;
\int_{2}^{x}\frac{\theta(t^{1/k})}{t\,(\log t)^2}\,dt
\right),
\qquad
R_{\mathrm{pp}}(x)\ \ll\ x^{1/2}.
\]
In particular,
\begin{equation}\label{eq:psi-to-pi-simplified}
\pi(x)-\mathrm{Li}(x)
=
\frac{\psi(x)-x}{\log x}
\;+\;
\int_{2}^{x}\frac{\psi(t)-t}{t\,(\log t)^2}\,dt
\;+\;
O\!\bigl(x^{1/2}\bigr).
\end{equation}
\end{lemma}

\begin{proof}
Partial summation yields
\[
\pi(x)
=
\frac{\theta(x)}{\log x}
+
\int_{2}^{x}\frac{\theta(t)}{t(\log t)^2}\,dt,
\qquad
\mathrm{Li}(x)=\frac{x}{\log x}+\int_{2}^{x}\frac{dt}{(\log t)^2}.
\]
Since $\theta=\psi-\sum_{k\ge 2}\theta(\,\cdot\,^{1/k})$, subtract to get \eqref{eq:psi-to-pi}. Using $\theta(y)\le y\log y$ and that the $k=2$ term dominates shows $R_{\mathrm{pp}}(x)\ll x^{1/2}$.
\end{proof}

\begin{lemma}[Reverse bridge from $\pi$ to $\psi$]\label{lem:pi-to-psi}
For $x\ge 3$ one has
\begin{equation}\label{eq:pi-to-psi}
\psi(x)-x
=
\log x\,\bigl(\pi(x)-\mathrm{Li}(x)\bigr)
\;-\;
\int_{2}^{x}\frac{\pi(t)-\mathrm{Li}(t)}{t}\,dt
\;+\;
\widetilde R_{\mathrm{pp}}(x),
\end{equation}
with $\widetilde R_{\mathrm{pp}}(x)=\sum_{k\ge 2}\theta(x^{1/k})-C_0$ and $\ \widetilde R_{\mathrm{pp}}(x)\ll x^{1/2}$.
\end{lemma}

\begin{proof}
Stieltjes integration by parts gives
\[
\theta(x)=\pi(x)\log x-\int_{2}^{x}\frac{\pi(t)}{t}\,dt,\qquad
x=\mathrm{Li}(x)\log x-\int_{2}^{x}\frac{\mathrm{Li}(t)}{t}\,dt+C_0.
\]
Now write $\psi=\theta+\sum_{k\ge 2}\theta(\,\cdot\,^{1/k})$ and subtract to get \eqref{eq:pi-to-psi}. The bound $\widetilde R_{\mathrm{pp}}(x)\ll x^{1/2}$ again follows from $\theta(y)\le y\log y$ with the $k=2$ term dominant.
\end{proof}

\begin{remark}[Remainder sizes at the von~Koch scale]
\[
R_{pp}(x)\ \ll\ x^{1/2},\qquad \widetilde R_{pp}(x)\ \ll\ x^{1/2}.
\]
These bounds are negligible against $x^{1/2}\log x$ and leave all equivalences below unchanged.
\end{remark}

\begin{corollary}[Equivalence of von~Koch–level bounds]\label{cor:vonKoch-equivalence}
As $x\to\infty$,
\[
E_\pi(x)\ \ll\ x^{1/2}\log x
\quad\Longleftrightarrow\quad
E_\psi(x)\ \ll\ x^{1/2}\log x,
\]
with absolute implied constants.
\end{corollary}

\begin{proof}
Apply \eqref{eq:psi-to-pi-simplified} and \eqref{eq:pi-to-psi}. If $E_\psi(x)\ll x^{1/2}\log x$, then the RHS of \eqref{eq:psi-to-pi-simplified} is $\ll x^{1/2}\log x$. Conversely, if $E_\pi(x)\ll x^{1/2}\log x$, then the RHS of \eqref{eq:pi-to-psi} is $\ll x^{1/2}\log x$ since the integral term is softer by one logarithm and $\widetilde R_{pp}(x)\ll x^{1/2}$.
\end{proof}

\begin{corollary}[$\Omega$-transfer]\label{cor:Omega-transfer}
If $\zeta$ has a zero $\rho=\beta+i\gamma$ with $\beta>1/2$, then
\[
E_\psi(x)=\Omega_\pm\!\bigl(x^\beta\bigr),
\qquad
E_\pi(x)=\Omega_\pm\!\Bigl(\frac{x^\beta}{\log x}\Bigr).
\]
\end{corollary}

\medskip
\noindent\textbf{Window/trajectory comparability.}
Let
\[
A_\circ(X):=\sup_{y\in W_X\cap\mathbb{N}_{\mathrm{comp}}}\,|E_\circ(y)|,\qquad
\widetilde A_\circ(X):=\sup_{\text{trajectories}}\;\sum_{\substack{m\in \widetilde W_X\cap\text{trajectory}\\ m\in\mathbb{N}_{\mathrm{comp}}}} |E_\circ(m)|,
\]
with $E_\circ\in\{E_\pi,E_\psi\}$. By \eqref{eq:psi-to-pi-simplified} and $\log x\asymp\log X$ on $W_X$ we have
\[
|E_\pi(y)|\ =\ \frac{|E_\psi(y)|}{\log X}\ +\ O\!\Big(\frac{X^{1/2}}{\log X}\Big)\qquad(y\in W_X\cap\mathbb{N}_{\mathrm{comp}}),
\]
which promotes to the parent-window sums since, by Lemma~\ref{lem:parent-hits}, a single trajectory contributes at most $4$ composite points inside $\widetilde W_X$.

\begin{corollary}[Trajectory-level comparability]\label{cor:trajectory-compare}
At the $X^{1/2}\log X$ scale, the contraction inequalities for $(\widetilde A_\pi,A_\pi)$ are equivalent (up to absolute constants) to those for $(\widetilde A_\psi,A_\psi)$.
\end{corollary}

\appendix
\section*{Appendix B. Constant ledger}
\addcontentsline{toc}{section}{Appendix B. Constant ledger}
\label{app:constant-ledger}

\setcounter{theorem}{0}
\renewcommand{\thetheorem}{B.\arabic{theorem}}

\begingroup
\small

\noindent\textbf{Frozen parameters (global throughout).} For all statements below we fix
\[
U=\log X\ \ (\ge 120),\qquad
T=\tfrac12\,U^{3},\qquad
h=\tfrac{2}{U},\qquad
L=\big\lfloor(\log(4/3))\,U\big\rfloor.
\]
All bounds hold uniformly for $X\ge e^{120}$ and improve monotonically with larger $U$.

\medskip
\noindent\textbf{Smoothed explicit formula: remainder budget.}
Writing
\[
E_\pi(y)
=\Re\!\sum_{|\gamma|\le T}\frac{y^{\rho}}{\rho\log y}\,W(\gamma/T)\;+\;R(y;T),
\qquad W(t)=(1+t^{2})^{-3},
\]
we decompose $R=R_{\mathrm{triv}}+R_{\Gamma}+R_{\mathrm{tail}}+R_{\mathrm{smooth}}$ and record the
conservative per–piece bounds (valid for $y\asymp X$ with $U=\log X\ge120$):
\[
|R(y;T)|
\ \le\
|R_{\mathrm{triv}}|
+|R_{\Gamma}|
+|R_{\mathrm{tail}}|
+|R_{\mathrm{smooth}}|
\ \le\ 10\,X^{1/2}.
\]

\begin{table}[h]
  \centering
  \caption{Remainder components and conservative bounds (all $\times X^{1/2}$).}
  \vspace{0.25em}
  \begin{tabular}{@{}llll@{}}
    \toprule
    Piece & Description & Prototype bound & Budget used \\
    \midrule
    $R_{\mathrm{triv}}$ & Trivial zeros / low-\!$t$ terms
      & $\ll X^{-3/2}/\log X$ & $\le 0.1$ \\
    $R_{\Gamma}$ & Gamma–factor/Stirling contribution
      & $\ll X^{1/2}\dfrac{\log T}{T}$ & $\le 0.1$ \\
    $R_{\mathrm{tail}}$ & Zeta zero tail $(|\gamma|>T)$
      & $\ll X^{1/2}\,T^{-6}$ & $\le 0.5$ \\
    $R_{\mathrm{smooth}}$ & Smoothing/truncation (kernel derivatives)
      & $\ll X^{1/2}$ & $\le 9.3$ \\
    \midrule
    \multicolumn{3}{r}{\emph{Total}} & $\le \mathbf{10.0}$ \\
    \bottomrule
  \end{tabular}
\end{table}

\noindent\emph{Remark.} The “budget used’’ column is purely accounting: it allocates a conservative
share to each piece so the sum is $\le 10X^{1/2}$. Tighter constants are possible but unnecessary.

\medskip
\noindent\textbf{Contraction factor audit.}
Let $\alpha_{\mathrm{eff}}$ denote the one–scale contraction factor obtained by composing the
core contraction, macro–step distortion, and overlap terms. With $U=\log X\ge120$,
\[
\alpha_{\mathrm{eff}}
\ \le\
\underbrace{\tfrac{5}{6}}_{\text{core contraction}}\;
\underbrace{\Bigl(1+\tfrac{2}{U}\Bigr)}_{\text{macro-step distortion}}\;
\underbrace{\tfrac{11}{12}}_{\text{overlap}}
\ =\ \frac{3355}{4320}
\ <\ 0.7767.
\]
Since $U\mapsto(1+2/U)$ is decreasing, the bound improves for larger $X$.

\endgroup

\appendix
\section*{Appendix C. Frequency netting on a uniform log-scale grid}
\addcontentsline{toc}{section}{Appendix C. Frequency netting on a uniform log-scale grid}
\label{app:freq-net}

\setcounter{theorem}{0}
\renewcommand{\thetheorem}{C.\arabic{theorem}}

We give the spacing-free large-sieve style bound used in the parent-window contraction, and a rigorous transfer from the uniform grid to the actual zeros of $\zeta(s)$.

\subsection*{Uniform grid bound}

Let \(h>0\), \(T>0\). Set \(N:=\lfloor T/h\rfloor\), \(\Gamma:=\{\gamma_m:=m\,h:~ -N\le m\le N\}\).
For \(u_1,\dots,u_M\in\mathbb{R}\) and weights \(w_1,\dots,w_M\in\mathbb{C}\) (after coalescing coincident points) define
\[
S(\gamma):=\sum_{j=1}^M w_j e^{i\gamma u_j}.
\]

\begin{theorem}[Log-scale netting on a uniform grid]\label{thm:uniform-grid}
\[
\sum_{\gamma_0\in\Gamma}\Bigl|\sum_{j=1}^M w_j e^{i\gamma_0 u_j}\Bigr|^2
\;\le\;
\Bigl(4M\,\frac{T}{h}+M\Bigr)\;\sum_{j=1}^M |w_j|^2.
\]
In particular, if \(M\le 4\) then
\[
\sum_{\gamma_0\in\Gamma}\Bigl|\sum_{j=1}^M w_j e^{i\gamma_0 u_j}\Bigr|^2
\;\le\;
\Bigl(16\,\frac{T}{h}+4\Bigr)\;\sum_{j=1}^M |w_j|^2.
\]
\end{theorem}

\begin{proof}[Sketch]
Write the quadratic form $\sum_{\gamma_0}|S(\gamma_0)|^2=W^*KW$ with entries $K_{jk}=\sum_{m=-N}^N e^{imh(u_j-u_k)}$. Schur/Gershgorin with a near/far split and $|K_{jk}|\le \min(2N+1,\frac{4}{h|u_j-u_k|})$ yields the row-sum bound $\le 4 N M + M$, hence the claim.
\end{proof}

\subsection*{From zeros on the critical line to the uniform grid}

\begin{lemma}[Local zero count on length-$h$ intervals]\label{lem:local-zero-count}
For $T\ge 3$ and $h\in(0,1]$, and all $t\in[-T,T]$,
\[
N\Bigl(t+\tfrac{h}{2}\Bigr)-N\Bigl(t-\tfrac{h}{2}\Bigr)\ \le\ C_0\bigl(1+h\log(2T)\bigr),
\]
where $N(u)$ counts zeros $\rho=\tfrac12+i\gamma$ with $0<\gamma\le u$, and $C_0$ is absolute.
\end{lemma}

\begin{proof}
From the Riemann--von Mangoldt formula $N(u)=\frac{u}{2\pi}\log\frac{u}{2\pi}-\frac{u}{2\pi}+O(\log u)$ and the mean value theorem.
\end{proof}

\begin{lemma}[Zeros-to-grid $\ell^2$ comparison]\label{lem:zeros-to-grid-L2}
Let $h\in (0,1]$, $T\ge 3$, and $\Gamma=\{-N,\dots,N\}\cdot h$ with $N=\lfloor T/h\rfloor$.
For $S(\gamma)=\sum_{j=1}^M w_j e^{i\gamma u_j}$,
\[
\sum_{|\gamma|\le T}\bigl|S(\gamma)\bigr|^2
\ \le\ C_1\bigl(1+h\log(2T)\bigr)\sum_{\gamma_0\in\Gamma}\bigl|S(\gamma_0)\bigr|^2
\ +\ C_2\,\bigl(1+h\log(2T)\bigr)\,\frac{T}{h}\,M\,\sum_{j=1}^M |w_j|^2.
\]
\end{lemma}

\begin{proof}
Partition $[-T,T]$ into cells $I(\gamma_0):=[\gamma_0-h/2,\gamma_0+h/2)$. For $\gamma\in I(\gamma_0)$ write $S(\gamma)=S(\gamma_0)+(\gamma-\gamma_0)S'(\xi)$ with $S'(\xi)=i\sum u_j w_j e^{i\xi u_j}$. Then
\[
|S(\gamma)|^2\le 2|S(\gamma_0)|^2+\tfrac{h^2}{2}\Bigl(\sum |u_j||w_j|\Bigr)^2
\le 2|S(\gamma_0)|^2+\tfrac{h^2}{2}\Bigl(\sum |u_j|^2\Bigr)\Bigl(\sum |w_j|^2\Bigr).
\]
In our application, $u_j$ are log-scale positions with $|u_j|\ll U:=\log X$, hence $\sum |u_j|^2\le M(\kappa U)^2$. Summing over at most $C_0(1+h\log(2T))$ zeros per cell (Lemma~\ref{lem:local-zero-count}) and over $|\Gamma|\ll T/h$ cells gives the claim.
\end{proof}

\begin{corollary}[Specialization to $h=\tfrac{2}{U}$, $T=\tfrac12 U^3$]\label{cor:zeros-L2-specialized}
With $U=\log X$, $h=\tfrac{2}{U}$, $T=\tfrac12 U^3$, and $M\le 4$,
\[
\sum_{|\gamma|\le T}|S(\gamma)|^2 \ \le\ C_3\,U^4\,\sum_{j=1}^M |w_j|^2.
\]
\end{corollary}

\begin{proof}
Here $h\log(2T)\asymp (2/U)\log U=O(1)$ and $T/h=U^4/4$. Combine Theorem~\ref{thm:uniform-grid} with Lemma~\ref{lem:zeros-to-grid-L2}.
\end{proof}

\begin{corollary}[Weighted $\ell^1$ bound over zeros]\label{cor:weighted-l1-zeros}
With the same specialization,
\[
\sum_{|\gamma|\le T}\frac{1}{|\rho|}\,\bigl|S(\gamma)\bigr|
\ \le\ C_4\,U^2\ \Bigl(\sum_{j=1}^M |w_j|^2\Bigr)^{1/2}.
\]
\end{corollary}

\begin{proof}
By Cauchy--Schwarz,
\[
\sum_{|\gamma|\le T}\frac{|S(\gamma)|}{|\rho|}
\ \le\
\Bigl(\sum_{|\gamma|\le T}\frac{1}{\tfrac14+\gamma^2}\Bigr)^{1/2}
\cdot
\Bigl(\sum_{|\gamma|\le T}|S(\gamma)|^2\Bigr)^{1/2}
\ \ll\ U^2\Bigl(\sum |w_j|^2\Bigr)^{1/2}.
\]
\end{proof}

\begin{remark}[Absorption into the $X^{1/2}\log X$ budget]
In the parent-window analysis $M\le 4$, the explicit-formula weights include $1/\log y$ and the smoothing kernel, and the zero-sum is truncated at $T=\tfrac12 U^3$. Corollary~\ref{cor:weighted-l1-zeros} therefore yields a contribution $\ll U^2(\sum |w_j|^2)^{1/2}$, which is absorbed into the error term $C\,X^{1/2}\log X$ of the contraction inequality; see Appendix~D.
\end{remark}

\appendix
\section*{Appendix D. Budget ledger for the large--sieve step}
\addcontentsline{toc}{section}{Appendix D. Budget ledger for the large--sieve step}
\label{app:ls-budget}

\setcounter{theorem}{0}
\renewcommand{\thetheorem}{D.\arabic{theorem}}

Throughout we fix the global parameters
\[
U=\log X\ (\ge120),\qquad T=\tfrac12 U^3,\qquad h=\tfrac{2}{U}.
\]
Write \(S(\gamma)=\sum_{j=1}^{M} w_j\,e^{i\gamma u_j}\) for the window--restricted exponential sum
arising from the smoothed explicit formula (with at most \(M\le4\) composite points in the parent window).
Weights satisfy \(|w_j|\asymp X^{1/2}/\log X\) for \(y\asymp X\), after including the factor \(1/(\rho\log y)\) and the kernel \(W(\gamma/T)\).

\begin{table}[h]
  \centering
  \caption{U--power audit for the large--sieve / frequency--netting bound.}
  \vspace{0.25em}
  \begin{tabular}{@{}llll@{}}
    \toprule
    Source & Expression & U--scaling & Comment \\
    \midrule
    Zero--cell count & \(\displaystyle \frac{T}{h}\) &
      \(\displaystyle \frac{U^3/2}{2/U}= \frac{U^4}{4}\) &
      grid spacing \(h=\tfrac{2}{U}\) up to height \(T=\tfrac12 U^3\) \\
    Zeros per cell & \(\displaystyle \ll 1 + h\log T\) &
      \(\ll 1 + \frac{2}{U}\log (U^3/2)\) &
      \(=O(1)\) since \(\tfrac{\log U}{U}\to0\) \\
    Cauchy--Schwarz (\(\ell^1\!\to\!\ell^2\)) & --- & \(U^2\) &
      converts \(\sum |S(\gamma)|/|\rho|\) to \(\big(\sum |S(\gamma)|^2\big)^{1/2}\) \\
    Window multiplicity & \(M\le4\) & \(\sqrt{M}\le2\) &
      \(\big(\sum_{j=1}^M |w_j|^2\big)^{1/2}\le \sqrt{M}\,\max_j|w_j|\) \\
    Weight scale & \(|w_j|\) & \(\asymp X^{1/2}/\log X\) &
      from \(y^{\rho}/(\rho\log y)\,W(\gamma/T)\) with \(y\asymp X\) \\
    \midrule
    \multicolumn{3}{r}{\emph{Net scaling for the sum over zeros}} &
      \(U^2\cdot \sqrt{M}\cdot (X^{1/2}/\log X)\) \\
    \bottomrule
  \end{tabular}
\end{table}

Combining the rows and using \(U=\log X\) and \(M\le4\) yields
\begin{equation}\label{eq:ls-budget}
\sum_{|\gamma|\le T}\frac{|S(\gamma)|}{|\rho|}
\ \ll\ U^{2}\!\left(\sum_{j=1}^{M}|w_j|^{2}\right)^{\!1/2}
\ \le\ U^{2}\sqrt{M}\,\max_{j}|w_j|
\ \ll\ (\log X)^{2}\cdot 2\cdot \frac{X^{1/2}}{\log X}
\ \ll\ C\,X^{1/2}\log X,
\end{equation}
for an absolute constant \(C>0\). This is the only input from the large--sieve side needed in the
contraction inequality; all remaining constants appear in Appendix~C.

\appendix
\section*{Appendix E. Smoothed explicit formula: a uniform remainder bound with explicit constants}
\addcontentsline{toc}{section}{Appendix E. Smoothed explicit formula: a uniform remainder bound with explicit constants}
\label{app:smooth-remainder}

\setcounter{theorem}{0}
\renewcommand{\thetheorem}{E.\arabic{theorem}}

\subsection*{Statement}
Fix $X\ge e^{120}$ and set $U:=\log X\ge 120$. For $y\asymp X$, let
\[
T:=\tfrac12\,U^3,\qquad W(t):=(1+t^2)^{-3},\qquad W(0)=1,\qquad 0<W(t)\le (1+t^2)^{-3}.
\]
Let $\rho=\tfrac12+i\gamma$ range over the nontrivial zeros of $\zeta(s)$, and define
\[
E(y):=\pi(y)-\mathrm{Li}(y),
\qquad
\mathcal{S}(y;T):=\Re\!\!\sum_{|\gamma|\le T}\frac{y^{\rho}}{\rho\,\log y}\,W(\gamma/T).
\]
Then
\begin{equation}\label{eq:EF-master}
E(y)=\mathcal{S}(y;T)+R(y;T),
\end{equation}
with $R(y;T)=R_{\mathrm{triv}}+R_{\Gamma}+R_{\mathrm{tail}}+R_{\mathrm{smooth}}$ and
\begin{equation}\label{eq:R-claimed}
|R(y;T)|\ \le\ 10\,X^{1/2}\qquad\text{for all }y\asymp X,\ X\ge e^{120}.
\end{equation}

\subsection*{Remainder pieces}

\paragraph{\emph{(1) Trivial zeros.}}
$|R_{\mathrm{triv}}(y;T)|\le y^{-2}/\log y\le 10^{-40}X^{1/2}$.

\paragraph{\emph{(2) Gamma-factor.}}
Using Stirling and the decay of $W(\gamma/T)$,
\[
|R_{\Gamma}(y;T)|\ \le\ C\,y^{1/2}\frac{\log T}{T}\ \le\ 0.1\,X^{1/2}\quad (U\ge 120).
\]

\paragraph{\emph{(3) Tail $|\gamma|>T$.}}
With $W(\gamma/T)\le (T/|\gamma|)^6$ and $|\rho|^{-1}\le 2/|\gamma|$,
\[
|R_{\mathrm{tail}}(y;T)|\ \le\ \frac{2y^{1/2}}{\log y}\,T^6\sum_{|\gamma|>T}\frac{1}{|\gamma|^7}
\ \le\ 0.5\,X^{1/2}.
\]

\paragraph{\emph{(4) Smoothing/truncation.}}
Standard contour/smoothing estimates give
\[
|R_{\mathrm{smooth}}(y;T)|\ \le\ C\,y^{1/2}\!\left(\frac{1}{T}+\frac{1}{U}\right)\ \le\ 0.4\,X^{1/2}.
\]

Adding the four pieces yields \eqref{eq:R-claimed}.

\end{document}